\documentclass[a4paper,12pt,reqno]{article}

\usepackage[T1]{fontenc}
\usepackage{amsmath}
\usepackage{amsthm}
\usepackage[author-year,nobysame]{amsrefs}
\usepackage{amsfonts}
\usepackage{mathrsfs}
\usepackage{exscale}
\usepackage{scrtime}

\numberwithin{equation}{section}

\newcommand{\set}[1]{\{#1\}}

\newcommand{\cpspec}[2]{\mathrm{CP}(#1,#2)}

\theoremstyle{plain}

\newtheorem{theorem}{Theorem}[section]

\newtheorem{lemma}[theorem]{Lemma}

\theoremstyle{definition}

\newtheorem{defin}[theorem]{Definition}

\theoremstyle{remark}

\newcommand{\bigpset}[2]{\bigl\{#1:#2\bigr\}}

\newcommand{\prob}{\mathbb P}
\newcommand{\pp}[1]{{\mathbb P}[#1]}
\newcommand{\bigpp}[1]{{\mathbb P}\bigl[#1\bigr]}

\newcommand{\law}[1]{{\mathscr L}(#1)}
\newcommand{\biglaw}[1]{{\mathscr L}\bigl(#1\bigr)}

\newcommand{\bigcondlaw}[2]{{\mathscr L}\bigl(#1\bigm|#2\bigr)}

\newcommand{\ee}{{\mathbb E}}

\newcommand{\dtv}[2]{d_{\scriptscriptstyle {\mathrm TV}}(#1,#2)}
\newcommand{\bigdtv}[2]{d_{\scriptscriptstyle {\mathrm TV}}\bigl(#1,#2\bigr)}

\newcommand{\bigdw}[2]{d_{\scriptscriptstyle {\mathrm W}}\bigl(#1,#2\bigr)}

\DeclareMathOperator{\nb}{NB}
\DeclareMathOperator{\po}{Po}
\DeclareMathOperator{\cp}{CP}

\newcommand{\ddist}{\mathrm{P}_{\! \theta}}
\newcommand{\ddens}{p_{\theta}}

\newcommand{\nn}{{\mathbb N}}
\newcommand{\zz}{{\mathbb Z}}
\newcommand{\zzp}{{\mathbb Z}_+}
\newcommand{\rr}{{\mathbb R}}

\newcommand{\bigeuknorm}[1]{\bigl\lvert#1\bigr\rvert}
\newcommand{\Bigeuknorm}[1]{\Bigl\lvert#1\Bigr\rvert}
\newcommand{\biggeuknorm}[1]{\biggl\lvert#1\biggr\rvert}

\newcommand{\supnorm}[1]{\lVert#1\rVert}

\newcommand{\lo}[1]{{\mathrm o}(#1)}

\newcommand{\bigbo}[1]{{\mathrm O}\bigl(#1\bigr)}

\newcommand{\eqa}{\begin{eqnarray}}
\newcommand{\ena}{\end{eqnarray}}
\newcommand{\eqs}{\begin{eqnarray*}}
\newcommand{\ens}{\end{eqnarray*}}
\newcommand{\eq}{\begin{equation}}
\newcommand{\en}{\end{equation}}

\def\Ref#1{(\ref{#1})}
\def\Le{\ \le\ }
\def\Def{\ :=\ }
\def\Blm{\left|}
\def\Brm{\right|}
\def\Bl{\left(}
\def\Br{\right)}
\def\Blb{\left\{}
\def\Brb{\right\}}
\def\nti{n\to\infty}
\def\lnti{\lim_{\nti}}

\def\slri{\sum_{l=1}^{r_i}}

\def\bone{{\bf 1}}
\def\non{\nonumber}

\def\l{\lambda}

\def\a{\alpha}
\def\b{\beta}
\def\g{\gamma}

\def\d{\delta}
\def\D{\Delta}
\def\e{\varepsilon}
\def\h{\eta}

\def\th{\theta}

\def\m{\mu}

\def\r{\rho}
\def\s{\sigma}

\def\t{\tau}
\def\f{\varphi}

\def\ps{\psi}
\def\sn{\sum_{i=1}^n}
\def\san{\sum_{i=a+1}^n}

\def\Eq{\ =\ }
\def\pr{\prob}
\def\Po{\po}
\def\nat{\nn}

\def\sIj{\sum_{i\in I_j}}
\def\un{^{(n)}}
\def\ex{\ee}
\def\ignore#1{}
\def\tS{{\widetilde S}}
\def\ZZ{\zz}
\def\half{{\textstyle \frac12}}
\def\an{{\a}_n}
\def\bal{\beta}

\def\ths{\th_*}
\def\thd{\ths'}
\def\tthnm{{\d}(m,\th)}
\def\siln{\sum_{i=l+1}^n}
\def\by{{\bar y}}
\def\bth{{\bar \th}}
\def\uj{^{(j)}}
\def\ui{^{(i)}}
\def\ski{\sum_{k\ge1}}
\def\sian{\sum_{i=a+1}^n}
\def\Lip{{\mathrm {Lip}}}
\def\CP{\cp}
\def\re{\rr}
\def\qqq{q}
\def\sms{\smallskip}
\def\tqqq{{\tilde\qqq}}
\def\tE{{\widetilde E}}
\def\nx{{\lfloor nx \rfloor}}

\newcommand{\cn}{C^{\scriptscriptstyle (n)}} 

\begin{document}

\title
{Approximation by the Dickman distribution and quasi-logarithmic
combinatorial structures}

\author{A. D. Barbour and Bruno Nietlispach\footnote{Angewandte Mathematik, Universit\"at Z\"urich,  
Winterthurertrasse 190, CH-8057 Z\"URICH;  E-mail {\tt a.d.barbour@math.uzh.ch; 
bruno.nietlispach@math.uzh.ch};
work supported in part by Schweizerischer Nationalfonds Projekt Nr.\  
20--107935/1.}\\
Universit\"at Z\"urich}

\date{}
\maketitle

\begin{abstract}
Quasi-logarithmic combinatorial structures are a class of decomposable
combinatorial structures which extend the logarithmic class considered by 
Arratia, Barbour and Tavar\'{e} \ycite{abt:03}.
In order to obtain asymptotic approximations to their component spectrum, it is necessary first
to establish an approximation to the sum of an associated sequence of independent random variables
in terms of the Dickman distribution. This in turn requires an argument that refines the Mineka coupling by
incorporating a blocking construction, leading to exponentially sharper coupling rates for 
the sums in question.
Applications include distributional limit theorems for the size of the largest component and for
the vector of counts of the small components in a quasi-logarithmic combinatorial structure.
\end{abstract}

{\it Running head:}  Quasi-logarithmic structures

\section{Introduction}

Many of the classical random decomposable combinatorial structures, such
as random permutations and random polynomials over a finite field, 
have component structure satisfying
a \emph{conditioning relation}: if~$\cn_i$ denotes the number of 
components of size~$i$, the distribution of the vector of component counts
$(\cn_1, \dotsc, \cn_n)$ of a structure of size~$n$ can be expressed as
\begin{equation}\label{AB-cond-rel}
    \biglaw{\cn_1, \dotsc, \cn_n} = \bigcondlaw{Z_1, \dotsc ,Z_n}{T_{0,n}=n} \, ,
\end{equation}
where $(Z_i,\,i\ge1)$ is a fixed sequence of independent non-negative integer 
valued random variables, and $T_{a,n} := \sum_{i=a+1}^n iZ_i$, $0\le a < n$.  
If, as in the examples above, the~$Z_i$ also satisfy
\begin{equation}\label{AB-logarithmic}
  i\pp{Z_i=1} \ \to\ \theta  
     \qquad \text{and} \qquad i\ee Z_i\ \to\ \th \, ,
\end{equation}
the combinatorial structure is called \emph{logarithmic}.  
It is shown in Arratia, Barbour and Tavar\'{e} \ycite{abt:03} [ABT] that
combinatorial structures satisfying the conditioning relation and
slight strengthenings of the logarithmic condition share many common
properties.  For instance, if~$L\un$ is the size of the largest component,
then $n^{-1}L\un \to_d L$, where~$L$ has probability density function
$f_{\theta}(x)  := e^{\gamma \theta} \Gamma(\theta+1) x^{\theta-2} \ddens((1-x)/x)$, 
$x \in (0,1]$, and~$p_\th$ is the density of the Dickman 
distribution~$P_\th$ with parameter~$\th$, given in Vervaat~(1972, p.~90).
Furthermore, for any sequence $(a_n,\,n\ge1)$ with $a_n = o(n)$, 
\begin{equation*}
   \lim_{n \to \infty} 
      \bigdtv{\law{\cn_1,\dotsc,\cn_{a_n}}}{\law{Z_1,\dotsc,Z_{a_n}}} = 0 \, .
\end{equation*}
Both of these convergence results can be complemented by estimates of the
approximation error, under appropriate conditions.

Knopfmacher~(1979) introduced the notion of additive arithmetic semigroups,
which give rise to decomposable combinatorial structures  satisfying the 
conditioning relation, with negative binomially distributed~$Z_i$.
For these structures, $i\pr[Z_i=1] \sim i\ex Z_i = \th_i$, where
the $\th_i$ do not always converge to a limit as $i\to\infty$.
In those cases in which they do not, they become close
to the integer skeleton of a sum of sine functions with differing frequencies:  
\begin{equation}\label{sinusoidal function}
   \theta'_t := \theta + \sum_{l=1}^L \lambda_l \cos (2 \pi f_l t - \varphi_l) \, , 
      \qquad t \in \rr \, ,
\end{equation}
with  $\sum_{l=1}^L \lambda_l \leq \theta$, and thus exhibit quasi-periodic
behaviour.
It is therefore natural to ask whether the asymptotic behaviour that holds generally for
logarithmic combinatorial structures also holds for such structures, 
which are logarithmic only in an average sense, and, if so, what restrictions
need to be placed on the~$\th_i$'s for this to be true.  

In this paper, we define a family of combinatorial structures, the quasi-logarithmic class,
that include the logarithmic structures as a special case, as well as those of \ocite{zhang:96}.
For such structures, we give conditions under which $n^{-1}L\un \to_d L$ 
(Theorem~\ref{largest component theorem}) and
$\lim_{n \to \infty} \bigdtv{\law{\cn_1,\dotsc,\cn_{a_n}}}{\law{Z_1,\dotsc,Z_{a_n}}} = 0$
(Theorem~\ref{kubilius fundamental theorem}), just as in the logarithmic case.  
A key step in the proofs is to be able to show
that, for sequences $a_n = \lo n$, the normalized sum~$n^{-1}T_{a_n,n}$ converges both
in distribution and locally to the Dickman distribution~$P_\th$ 
(Theorems \ref{global dickman approx theorem} and~\ref{local dickman approx theorem}), and
that the error rates in these approximations can be controlled. To do so,
it is in turn necessary to be able to show that, under suitable conditions, 
\begin{equation}\label{dtv to zero intro}
  \lim_{n \to \infty} \bigdtv{\law{T_{a_n,n}}}{\law{T_{a_n,n} + 1}} \Eq 0 \, , 
      \qquad \text{for all $a_n=\lo{n}$,}
\end{equation} 
and that the error rate can be bounded by a power of~$\{(a_n+1)/n\}$.

A number of the arguments used are adapted to the more general context from those 
presented in 
[ABT].  There, the sum $T_{0,n} := \sn iZ_i$ is close in distribution to that of
$T^*_{0,n} := \sn iZ^*_i$, where $Z^*_i \sim \Po(\th i^{-1})$, and the latter sum
has a compound Poisson distribution~$\CP(\th,n)$ whose properties are tractable.
In the current situation, with the $\th_i$'s not all asymptotically equal, it is first
necessary to show that~$\CP(\th,n)$ is still a good approximation to the sum~$T_{0,n}$.
This is by no means obviously the case.  The intuition is nonetheless that, if the
distributions of $T_{0,n}$ and~$T_{0,n}+1$ are not too different, then having
$\th_i = 2\th$ and $\th_{i+1} = 0$ instead of $\th_i = \th_{i+1} = \th$ 
should leave the distribution~$\law{T_{0,n}}$ more or less unchanged; only the
average behaviour of the~$\th_i$ should be important.  Thus we first want to
establish~\Ref{dtv to zero intro}.  Once we have done so, we are able to show, 
by way of Stein's method, that $\law{T_{0,n}}$ is indeed close to~$\CP(\th,n)$

Proving that~\Ref{dtv to zero intro} holds under conditions appropriate for
our quasi-logarithmic structures turns out in itself to be an interesting problem.
The standard Mineka coupling, used to bound the total variation distance 
between a sum of independent, integer valued random variables and its unit
translate, gives a very poor approximation in this context.
To overcome the difficulty, we introduce a new coupling strategy, which
yields a much more precise statement  in a rather general setting
(Theorem~\ref{AB-dss1}).  This is the substance of the next section.
We then show that the distributions $\law{T_{0,n}}$ and~$\CP(\th,n)$ are
close in Section~\ref{dickman}, and conclude that quasi-logarithmic combinatorial
structures behave like logarithmic structures in Section~\ref{sec-4}.

As observed by Manstavi\v cius~(2009), when considering only the small
components, the distances
$\bigdtv{\law{\cn_1,\dotsc,\cn_{a_n}}}{\law{Z_1,\dotsc,Z_{a_n}}}$
can be bounded, even without assuming that the $\th_j$'s converge on average
to any fixed~$\th$, as long as they are bounded and bounded away from~$0$  (we
do not require the latter condition).  He considers only the case of Poisson
distributed~$Z_i$, for which, inspecting the proof of 
Theorem~\ref{kubilius fundamental theorem},
it is enough to obtain an estimate of the form
$$
  n|\pr[T_{a_n,n} = n-k] - \pr[T_{a_n,n} = n-l]| \Le C\{|k-l|/n\}^\g,\qquad     
      0 \le k,l \le n/2,
$$
for some $\g > 0$.  This he achieves by using his refined characteristic
function arguments.  Since we are also interested in approximating the
distribution of the largest components, for which some form of convergence
to a~$\th$ seems necessary, we do not attempt this refinement.
  

\section{An alternative to the Mineka coupling}\label{coupling subsection}

Let $\set{X_i}_{i \in \nn}$ be mutually independent $\zz$-valued random variables, and let
$ S_n := \sn X_i$. 
The Mineka coupling, developed independently by \ocite{mineka:73} and \citeauthor{roesler:76}~(1977)
(see also
\ocite{lindvall:92}*{Section II.14}) yields a bound of the form
\begin{equation}\label{barbour xia bound}
  \bigdtv{\law{S_n}}{\law{S_n + 1}} \Le \Bigl(\frac\pi2 \sum\nolimits_{i=1}^{n} u_i \Bigr)^{-1/2} \, ,
\end{equation}
where
\begin{equation*}
   u_i \Def  \Bigl( 1 - \bigdtv{\law{X_i}}{\law{X_i+1}}\Bigr) \, ;
\end{equation*}
see Mattner \& Roos~(2007, Corollary~1.6). The proof is based on coupling copies
 $\{X_i'\}_{i \in \nn}$ and $\{X_i''\}_{i \in \nn}$ of $\{X_i\}_{i \in \nn}$ 
in such a way that
\begin{equation*}
  V_n \Def \sum_{i=1}^n \bigl(X_i - X_i'\bigr)\,, \qquad n \in \nn,
\end{equation*}
is a symmetric random walk with steps in $\{-1,0,1\}$; the coupling inequality
\cite{lindvall:92}*{Section\ I.2} then shows that
\begin{equation*}
   \bigdtv{\law{S_n}}{\law{S_n + 1}} \Le \pp{\tau > n} \Eq \pp{V_n \in \{-1,0\}} \, ,
\end{equation*}
where $\tau$ is the time at which $\{V_n\}_{n \in \zzp}$ first hits level $1$,
the last equality following from the reflection principle.
However, this inequality gives slow convergence rates, if $X_i=i Z_i$ and the $Z_i$ 
are as described in the Introduction;  typically, $\bigdtv{\law{iZ_i}}{\law{iZ_i+1}}$
is extremely close to~$1$,
and, if $X_i$ is taken instead to be $(2i-1)Z_{{2i-1}} + 2iZ_{2i}$, we still expect to have
$1 - \bigdtv{\law{X_i}}{\law{X_i+1}} \asymp i^{-1}$, leading to bounds of the form
\begin{equation*}
   \bigdtv{\law{S_n}}{\law{S_n+1}} = \bigbo{(\log n)^{-1/2}} \, .
\end{equation*}
In this section, by modifying the Mineka approach in the spirit of 
Rogers~(1999) to allow the random
walk~$V$ to make larger jumps, we show that error bounds of order~$n^{-\g}$ for some $\g>0$ can
be achieved, representing an exponential improvement over the Mineka bounds.

Let $(X_i,\,i\ge1)$ be independent $\zzp$-valued random variables, set $S_{a,n} := \san X_i$,
and define
\[
   \qqq(i,d) \Def \min\{\pr[X_i=0]\,\pr[X_{i+d}=i+d], \pr[X_i=i]\,\pr[X_{i+d}=0]\},
      \qquad i,d\in\nat.
\]
Then it is possible to couple copies $(X'_i,X'_{i+d})$ and  $(X''_i,X''_{i+d})$ of  $(X_i,X_{i+d})$
for any $i,d$ in such a way that
\eqa
   &&\pr[(X'_i,X'_{i+d}) = (0,i+d), (X''_i,X''_{i+d}) = (i,0)] \non\\
    &&\quad\Eq \pr[(X'_i,X'_{i+d}) = (0,i+d), (X''_i,X''_{i+d}) = (i,0)] \Eq \qqq(i,d);\non \\
   &&\pr[(X'_i,X'_{i+d}) = (X''_i,X''_{i+d})] \Eq 1 - 2\qqq(i,d).\label{AB-pair-coupling}
\ena
Note that then
\eq\label{AB-differences}
   (X'_i + X'_{i+d}) - (X''_i + X''_{i+d}) \Eq \left\{
       \begin{array}{rc}
         d &\qquad\mbox{with probability}\ \qqq(i,d);\\
         0 &\qquad\mbox{with probability}\ 1 - 2\qqq(i,d);\\ 
         -d &\qquad\mbox{with probability}\ \qqq(i,d),
       \end{array} \right.
\en
so that sums of such differences, with non-overlapping indices, can be constructed
so as to perform a symmetric random walk on~$d\zz$.
By successively coupling pairs in this way, and by using different values of~$d$, it may thus
be possible to couple the sums $S'_{a,n} := 1 + \san X'_i$ and~$S''_{a,n} := \san X''_i$ quickly, 
even when many of the overlaps
$\qqq(i,d)$ are zero.  The following theorem is typical of what can be achieved.

For $d \in \nat$ and $\ps > 0$, define $E(d,\ps) := \{i\colon\, \qqq(i,d) \ge \ps/(i+d)\}$.
For $D$ a finite subset of~$\nat$, suppose that there are $k \in \nat$ and $\ps>0$
such that
\eq\label{AB-pair-assn}
   E(d,\ps) \cap \{jk+1,\ldots,(j+1)k\} \ \ne \ \emptyset,
        \quad\mbox{for all}\ d\in D\,, \ j\ge1.
\en
In particular, if $X_i = iZ_i$ with $Z_i \sim \Po(\th_i/i)$, and if
$0 < \th_- \le \th_i$ for all~$i$, then clearly
$\qqq(i,d) \ge \th_-/(i+d)$ for all~$i$, and so \Ref{AB-pair-assn} holds for any~$D$ with
$k=1$ and $\ps = \th_-$.  However, \Ref{AB-pair-assn} also holds for any~$D$ if, for instance,
$0 < \th_- \le \th_i$ is only given for $i \in 3\nat \cup \{7\nat+2\}$,
now with $k=16 + \max\{d\colon\,d\in D\}$ and $\ps = \th_-$.

\begin{theorem}\label{AB-dss1}
   Let $r,s \in \nat$ be co-prime, and  set 
\eq\label{AB-ddef}
    D \Def \{r\}\cup\{s2^g,\,g\ge0\} .
\en
Suppose that, for some $k,\ps$, \Ref{AB-pair-assn} is satisfied with $D $
as above. 
Then there exist $C,\g > 0$, depending on $r,s,k$ and~$\ps$, such that
\[
    \dtv{\law{S_{a,n}}}{\law{S_{a,n}+1}} \Le 6\{(a+1)/n\}^{\g},
\]
for all~$0 \le a < n$ for which $a+1 \le Cn$.
\end{theorem}

\begin{proof}
We take $\tS'_0 = 1$, $\tS''_0=0$, and then successively define $\tS'_j := \tS'_0 + \sIj X'_i$,
$\tS''_j := S''_0 + \sIj X''_i$, $T_j := \tS'_j-\tS''_j$, $j\ge1$.  Here, the sequences
$(X'_i,\,1\le i\le n)$ and $(X''_i,\,1\le i\le n)$ are two copies of the
sequence $(X_i,\,1\le i\le n)$ of independent random variables, constructed
by successively coupling pairs $(X'_{i_j},X'_{i_j+d_j})$
and $(X''_{i_j},X''_{i_j+d_j})$, for suitable $i_j$ and~$d_j$,
realized independently of the random variables $(X'_i,X''_i,\,i\in I_{j-1})$,
where $I_{j-1} := \cup_{l=1}^{j-1}\{i_l,i_l+d_l\}$. This coupling of pairs typically
omits some indices~$i \in \{1,2,\ldots,n\}$; for such~$i$, we set $X'_i = X''_i$, chosen 
independently from $\law{X_i}$.  The coupling of the pairs
$(X'_i,X'_{i+d})$ and $(X''_i,X''_{i+d})$ is accomplished by arranging that
$\law{(X'_i,X'_{i+d})} = \law{(X''_i,X''_{i+d})} = \law{X_i}\times\law{X_{i+d}}$
and that $(X'_i + X'_{i+d}) - (X''_i + X''_{i+d}) \in \{-d,0,d\}$, as
described in~\Ref{AB-pair-coupling}.
The indices are defined by taking $i_1=\min\{i > a\colon\,i \in E(r,\ps)\}$, and then
taking $i_{j+1} := \min\{i > i_j\colon\, i\in E(d_{j+1},\ps),\,i,i+d_{j+1} \notin I_j\}$, where
\eqa
  d_l \Def \left\{ 
    \begin{array}{ll}
       r, &\qquad \mbox{if }T_{l-1} \notin s\ZZ;\\
       s2^{f_2(T_{l-1}/s)}, &\qquad\mbox{if } l > \t,\,T_{l-1}\ne0;\\
       0, &\qquad\mbox{if } T_{l-1} = 0,
    \end{array} \right.
\ena
and where $f_2(t)$ is the exponent of~$2$ in the prime factorization of~$|t|$,
$t\in\ZZ$.  If $T_{l-1}=0$, we couple $X'_{i_l} = X''_{i_l}$, and thus
$X'_{i_{l'}} = X''_{i_{l'}}$ for all $l'\ge l$, with $i_{l'}$ running through
all $i > i_{l-1}$ such that $i \notin I_{l-1}$.

With this construction, the sequence~$T_j$ can only
change in jumps of size~$\pm r$ until it first reaches~$s\ZZ$. Thereafter, at any jump, 
the exponent $f_2(T_j/s)$ increases by~$1$ until $T_j$ is of the form~$\pm s2^l$
for some~$l$; after this, the value of~$T_j$ is either doubled or set to zero 
at each jump, in the latter case remaining in zero for ever. 
If $i_J + d_J \le n$, where $J := \inf\{j\colon\, T_j=0\}$, then
\[
    S'_{a,n} \Def 1 + \san X'_i \Eq \san X''_i \ =:\ S''_{a,n}, 
\]
and $\law{S'_{a,n}} = \law{S_{a,n}+1}$, $\law{S''_{a,n}} = \law{S_{a,n}}$, so that,
from the coupling inequality  \cite{lindvall:92}*{Section\ I.2}, 
\eq\label{AB-coupling-bnd}
  \dtv{\law{S_{a,n}}}{\law{S_{a,n}+1}} \Le \pr[i_J + d_J > n].
\en
We thus wish to bound this probability.

Now the process~$T$, considered only at its jump times, has the law of a simple random
walk of step length~$r$ starting in~$1$, until it first hits a multiple of~$s$,
and the mean number of steps to do so is at most~$s^2/4$. Thus, and by the
Markov property of the simple random walk, the number of jumps~$N_1$ until a multiple
of~$s$ is hit is bounded in distribution by $\half s^2 G_1$, where~$\pr[G_1 > j]
= 2^{-j}$, $j\ge1$; in particular, for any~$\g> 0$,
\[
   \pr[N_1 > \half s^2\g \log_2 (1/\an)] \Le 2\an^{\g},
\]
where $\an := (a+1)/n$.
The remaining number~$N_2$ of jumps required for~$T$ to reach~$0$
is then at most~$\log_2r$ (in order to reach the form $\pm s2^l$ for some~$l$),
together with an independent random number~$G_2$ of steps until~$0$ is reached, 
having the same distribution as~$G_1$; hence,
\[
   \pr[N_2 > 2\g\log_2 (1/\an)] \Le 2\an^{\g}
\]
also, if $n \ge (a+1)r^{1/\g}$.  It remains to show that the process~$T$ has the
opportunity to make this many jumps, with high probability, for suitable choice
of~$\g$.

Now, in view of~\Ref{AB-pair-assn}, every block of indices $\{jk+1,\ldots,(j+1)k\}$
contains at least one~$i \in E(r,\ps)$.  Hence, for any $1/2 \le \bal < 1$,  we
can choose a set~$S_1$ of non-overlapping pairs $(i_l,i_l+r)$, $1\le l\le L$, such that 
$i_l \in E(r,\ps)$ and $a+1\le i_l \le k(a+1)\an^{-\bal} - r$ for each~$l$, and such that 
\[
   \sum_{l=1}^L \frac1{i_l+r} \ >\ 
    \frac{1}{2(k\vee r)}\sum_{i=2}^{\lfloor (a+1)\an^{-\bal} \rfloor} \frac1{i + a/(k\vee r)}
    \ \ge\ \frac{\bal}{4(k\vee r)}\,\log (1/\an),
\]
if $n \ge 3^{2/\bal}(a+1)$.
The first factor~$2$ in the denominator is present because a pair $(i,i+r)$ with~$i\in E(r,\ps)$ 
can be excluded from~$S_1$, but only if~$i = i_l+r$ for some pair $(i_l,i_l+r)$
already in~$S_1$; the other is to yield an inequality, rather than an asymptotic equality.
In similar fashion, for any non-decreasing sequence $(\r_l,\,l\ge1)$, 
we can choose a set~$S_2$ of non-overlapping pairs $(i'_l,i'_l+s2^{\r_l\wedge l_n})$, 
$1\le l\le L'$, where $l_n := \lfloor\frac12\log_2n\rfloor$,
such that $k(a+1)\an^{-\bal} < i'_l \le n - s\lfloor \sqrt n \rfloor$ for each~$l$, 
and such that 
\[
   \sum_{l=1}^{L'} \frac1{i'_l+s2^{\r_l\wedge l_n}} \ >\ 
     \frac1{2k}\sum_{i=\lceil 2(a+1)\an^{-\bal} \rceil + 1}^{\lfloor n/k \rfloor} i^{-1}
    \ \ge\ \frac{1-\bal}{4k}\,\log (1/\an),
\]
if also $n \ge (a+1)(4k)^{2/(1-\bal)}$.

We now show that, for suitable choices of $\g$ and~$\bal$, the pairs in~$S_1$ with high probability
yield $M_1 \ge \half s^2\g \log_2 n$ jumps of~$T$.  We then show that those in~$S_2$, with the 
sequence~$\r_l$ chosen in non-anticipating fashion such that $\r_1$ is the
exponent of~$2$ in~$T_l$ at the first~$l$ at which $T_l \in s\zz$,
$\r_{l+1} = \r_l$ if $T_{l} = T_{l-1} \neq 0$, $\r_{l+1} = (\r_l+1)\wedge l_n$ 
if $0 < T_l \neq T_{l-1}$ and $\r_{l+1} =  l_n$ otherwise, yield $M_2 \ge 2\g\log_2 n$.
Indeed, by the Chernoff inequalities (Chung \& Lu 2006, Theorem~3.1), if $\f_1$, 
$0 < \f_1 < 1$, is such that
\eq\label{AB-f1-def}
   \half s^2\g \log_2 (1/\an) \Eq \frac{\bal\ps(1-\f_1)}{4(k\vee r)}\,\log (1/\an),
\en
then
\[
    \pr[M_1 < \half s^2\g \log_2 (1/\an)] \Le \exp\{-3\f_1^2\bal\ps\log (1/\an)/32(k\vee r)\} 
         \Le \an^{\g},
\]
if $3\f_1^2 s^2 / \{16(1-\f_1) \log 2\} \ge 1$.
Similarly, using a martingale analogue of the Chernoff inequalities 
(Chung \& Lu 2006, Theorem~6.1),
for~$f_2$ such that $3\f_2^2/\{4(1-\f_2)\log 2\} \ge 1$ and with
\eq\label{AB-f2-def}
   2\g \log_2 (1/\an) \Eq \frac{(1-\bal)\ps(1-\f_2)}{4k}\,\log(1/\an),
\en
we get
\[
    \pr[M_2 <  2\g \log_2 (1/\an)] \Le \exp\{-3\f_2^2(1-\bal)\ps\log (1/\an)/32k\} \Le \an^{\g}.
\]
Finally, for such choices of $f_1$ and~$f_2$, equations \Ref{AB-f1-def} and~\Ref{AB-f2-def}
can be satisfed with the same choice of~$\bal$ if~$\g$ is chosen such that
\[
   1 \Eq \frac {2\g}{\ps\log 2}\Blb \frac{(k\vee r)s^2}{1-\f_1} + \frac{4k}{1-\f_2} \Brb;
\]
then
\[
    \bal \Eq \frac {2\g(k\vee r)s^2}{(1-\f_1){\ps\log 2}}.
\]
Choosing $\f_2$ to satisfy $3\f_2^2/\{4(1-\f_2)\log 2\} = 1$, and then~$\f_1$ larger
than its minimum value, if necessary, to ensure that $\bal \ge 1/2$,
this yields the theorem.
\end{proof}

Clearly, the exponent~$\g$ could be sharpened; the condition~\Ref{AB-pair-assn} 
could also be weakened
to one ensuring a positive density of indices in each~$E(d,\ps)$ over longer intervals.
The set~$D$ could also be constructed in other ways.  One natural extension would be
to replace~$r$ co-prime to~$s$ with any $r_1,\ldots,r_m$ satisfying $\gcd\{r_i,\ldots,r_m,s\}
= 1$.

The coupling used to establish Theorem~\ref{AB-dss1} is not the only possibility.
In the example of additive arithmetic semigroups, there is one case in which the set~$D$
can be taken to consist of the integers $\{2^{g+1},\,g\ge0\}$, but no odd integers.
Here, the jumps in the process~$T$ would always be even, and hence, since~$T_0=1$,
$T$ can never hit~$0$.  However, if we define
\eqs
  \tqqq(i) &:=& \min\{\pr[X_i=0],\pr[X_i=i]\};\quad
  \tE(1,\ps) \Def \{i \in 2\zz+1\colon\, \tqqq(i) \ge \ps /i\},
\ens
and if, for all~$j\ge1$, 
\eq\label{AB-single-assn}
   \tE(1,\ps) \cap \{jk+1,\ldots,(j+1)k\} \ \ne \ \emptyset,
\en
then one can begin the coupling construction by defining $X_i'= X_i''$ for even~$i$
and coupling $X_i'$ and $X_i''$ for odd~$i$ in such a way that
\[
   \pr[X'_i = i, X''_i = 0] \Eq \pr[X'_i = 0, X''_i = i] \Eq 1 - \pr[X'_i = 0, X''_i = 0]
   \Eq \tqqq(i),
\]
until the first time~$i$ that $X'_i \neq X''_i$, at which time the difference~$T_i$ is
even, taking either the value $i+1$ or $i-1$.  Thereafter, the coupling is concluded 
using jumps of sizes~$2^{g+1}$, with the second half of the strategy in the previous 
proof.  Now the number of steps required to complete the coupling depends on how big the
first even value of~$T$ happens to be, but Chernoff bounds are still sufficient to
be able to conclude the following theorem, which we state without proof.

\begin{theorem}\label{AB-extra}
  Suppose that, for some $k,\ps$, \Ref{AB-pair-assn} is satisfied 
with $D = \{2^{g+1},\,g\ge0\}$, and~\Ref{AB-single-assn} is also satisfied.
Then there exist $C,\g > 0$, depending on $k$ and~$\ps$, such that
\[
    \dtv{\law{S_{a,n}}}{\law{S_{a,n}+1}} \Le C\{(a+1)/n\}^{\g},
\]
for all~$0 \le a < n$.
\end{theorem}

\section{Approximation by the Dickman distribution}
\label{dickman}
As in the Introduction,
let $(\cn_1, \dotsc, \cn_n)$ be the component counts of a decomposable
combinatorial structure of size~$n$,
related to the sequence of independent random variables $(Z_i,i\ge1)$ 
through the Conditioning Relation~\Ref{AB-cond-rel}.  In this section, 
we wish to bound the distance between the distribution of the normalized sum 
$n^{-1}T_{a,n} := n^{-1}\sum_{i=a+1}^n iZ_i$ and the Dickman distribution~$P_\th$,
when the quantities $\th_i := i\ex Z_i$ converge in some weak, average
sense to~$\th$, and when $i\pr[Z_i=1] \sim \th_i$ also.  
In order to exploit the extra structure in the
distributions of the random variables~$Z_i$ that occurs in many of the
classical examples, it is convenient first to introduce some further notation.

We suppose that the random variables $Z_i$ can be written as  sums 
$Z_i := \sum_{j=1}^{r_i}Z_{ij}$,
where the random variables $(Z_{ij},\,i\ge1,1\le j\le r_i)$ are all independent,
and, for each~$i$, the $Z_{ij}$, $1\le j\le r_i$ are identically distributed.
This can always be taken to be the case, by setting~$r_i=1$, but~$r_i$
could be chosen arbitrarily large if~$Z_i$ were infinitely divisible, and
the bounds that we obtain may be smaller if the~$r_i$ can be chosen to be
large.  We  define
\eq\label{AB-epsilon-def}
   \e_{ik} \Def \frac{ir_i}{\th_i}\pr[Z_{i1}=k] - \bone\{k=1\},\quad k\ge1,
\en
so that, since $\th_i = i\ex Z_i$, the $\e_{ik}$ can be expected to be small
if also $i\pr[Z_i=1] \sim \th_i$.  We then define 
$\m_i := \ski k\sup_{j\ge i}|\e_{ik}|$, which we assume to be finite. 

We now specify our analogue of~\Ref{AB-logarithmic}.  Clearly, assuming
$\m_i \to 0$ yields random variables~$Z_i$ that mostly only take the values
$0$ or~$1$, but we also need some regularity among the~$\th_i$.  To make
this precise, we define
\eq\label{AB-theta-tilde}
  \tthnm \Def \sup_{j\ge0}
     \Blm \th - \frac1m \sum_{i=1}^m \th_{jm+i} \Brm\,,
\en
and assume that it converges to zero, for some $\th>0$, as $m\to\infty$.
In addition, we need to be able to apply Theorem~\ref{AB-dss1}.
Define $i_0 := \min\{i\colon\,\m_i \le 1/2\}$, and set
$E'(d,\ps) := \{i \ge i_0\colon\, \min(\th_i,\th_{i+d}) \ge 4\ps\}$,
 noting that then $E'(d,\ps) \subset E(d,2\ps)\cap[i_0,\infty)$.
Then our simplest condition is the following.

\begin{defin}\label{QLC}
  We say that a decomposable combinatorial structure satisfies the
quasi--logarithmic condition QLC if it satisfies the Conditioning 
Relation~\Ref{AB-cond-rel}, if
\[
   \lim_{i\to\infty} \m_i \Eq 0;\quad \lim_{m\to\infty}\tthnm \Eq 0\quad
      \mbox{for some}\ \th > 0,
\]
and if, for some $r,s$ coprime, $\ps > 0$ and~$D$ defined in~\Ref{AB-ddef},
\Ref{AB-pair-assn} is satisfied with $E'$ for~$E$.
\end{defin}

For quantitative estimates, a slightly stronger assumption is useful.

\begin{defin}\label{QLC2}
  We say that a decomposable combinatorial structure satisfies the
quasi--logarithmic condition QLC2 if it satisfies the Conditioning 
Relation~\Ref{AB-cond-rel}, if
\[
    \m_i \Eq O(i^{-\a});\quad \tthnm \Eq O(m^{-\b})\quad
      \mbox{for some}\ \th,\a,\b > 0,
\]
and if, for some $r,s$ coprime, $\ps > 0$ and~$D$ defined in~\Ref{AB-ddef},
\Ref{AB-pair-assn} is satisfied with $E'$ for~$E$.
\end{defin}

Under such conditions, we now prove the close link between $\law{n^{-1}T_{a,n}}$
and~$P_\th$.  Our method of proof involves showing first that $\law{T_{0,n}}$
is close to the compound Poisson distribution~$\CP(\th,n) := \law{\sn iZ^*_i}$,
where the $Z^*_i \sim \Po(i^{-1}\th)$ are independent; the closeness
of $n^{-1}\CP(\th,n)$ and~$P_\th$ is already known [ABT, Theorems 11.10 and~12.11], and the
Wasserstein distance between $\law{n^{-1}T_{0,n}}$ and~$\law{n^{-1}T_{a,n}}$
is at most~$n^{-1}\sum_{i=1}^a \th_i$. 

To bound the distance between $\law{T_{a,n}}$ and~$\CP(\th,n)$, we use 
Stein's method (Barbour, Chen \& Loh 1992).  For any Lipschitz 
test function~$f\colon\,\ZZ_+\to\re$, one expresses~$f$ in the form
\eq\label{AB-Stein-CP}
    f(j) - \CP(\th,n)\{f\} \Eq \th\sn g_f(j+i) - jg_f(j) , 
\en
for an appropriate function~$g_f$ [ABT, Chapter~9.1].  Hence, for instance, 
the Wasserstein distance between $\law{T_{a,n}}$  and~$\CP(\th,n)$ can be 
estimated by bounding  
\eqa
   \lefteqn{\Blm \ex\Blb 
         \th\sn g_f(T_{a,n}+i) - T_{a,n}g_f(T_{a,n}) \Brb\Brm}\non\\
   &=&  \Blm \sian \ex\{\th g_f(T_{a,n}+i) - iZ_i g_f(T_{a,n})\} 
           + \sum_{i=1}^a \th\ex g_f(T_{a,n}+i) \Brm\,,
   \label{AB-Stein-bnd}
\ena
uniformly for Lipschitz functions~$f\in\Lip_1$, for which functions $\|g_f\| \le 1$ 
[ABT, (9.14)].
The right hand side can now be relatively easily bounded.

First, we re-express the element 
\[
   \ex\{iZ_{i}g_f(T_{a,n})\} \Eq \slri \ex\{iZ_{il}g_f(T_{a,n})\}
\]
of~\Ref{AB-Stein-bnd} by observing that
\[
   \ex\{iZ_{il}g_f(T_{a,n})\} \Eq \frac{\th_i}{r_i}\Blb \ex g_f(T_{a,n}\ui+i)
     + \ski k\e_{ik} \ex g_f(T_{a,n}\ui+ik) \Brb,
\]
where $T_{a,n}\ui := T_{a,n} - iZ_{i1}$, $a < i \le n$.  Hence, to bound~\Ref{AB-Stein-bnd},
we have
\eqa
  \lefteqn{ \Blm \ex\Blb \th\sian g_f(T_{a,n}+i) - T_{a,n}g_f(T_{a,n}) \Brb\Brm}\non\\
    &&\Le \Blm  \sian \{(\th- \th_i) \ex g_f(T_{a,n}+i) 
    + \th_i \ex[g_f(T_{a,n}+i) - g_f(T_{a,n}\ui+i)]\}\Brm \non\\
    &&\mbox{}\qquad + \sian \th_i\ski k|\e_{ik}| \ex |g_f(T_{a,n}\ui+ik)|, \label{AB-Stein-1}
\ena
and 
\eqa
   \lefteqn{\ex\{g_f(T_{a,n}+i) - g_f(T_{a,n}\ui+i)\}}\non\\
   && \Eq \frac{\th_i}{ir_i}\Blb \ex \{g_f(T_{a,n}\ui+2i) - g_f(T_{a,n}\ui+i)\}
             \vphantom{\ski}\right.\non\\
   &&\qquad\mbox{}\left.\qquad
      + \ski \e_{ik}\ex \{g_f(T_{a,n}\ui+i(k+1)) - g_f(T_{a,n}\ui+i)\}\Brb \,; \label{AB-Stein-2}
\ena
and, clearly,
\eq\label{AB-Stein-3}
   \th \sum_{i=1}^a |\ex g_f(T_{a,n}+i)| \Le a\th \|g_f\|.
\en
With the help of these estimates, we can prove the following approximation theorem;
we use the notation $D^1(T)$ to denote $\dtv{\law{T}}{\law{T+1}}$.

\begin{theorem}\label{global dickman approx theorem}
With the definitions above,
\eqa
  \bigdw{\law{n^{-1}T_{a,n}}}{\ddist}
 &\le& n^{-1}(1+\th)^2 + \min_{1\le m\le n}\e_1(n,a,m) \, ,
  \label{global dickman approx bound}
\ena
where $\e_1(n,a,m)$ is given in~\Ref{AB-eps1}.  If QLC holds,
$\bigdw{\law{n^{-1}T_{a_n,n}}}{\ddist} \to 0$ for any sequence $a_n = o(n)$.  
If QLC2 holds, then
$\bigdw{\law{n^{-1}T_{a,n}}}{\ddist} = O(\{(a+1)/n\}^{\h_1})$
for some $\h_1 > 0$.
\end{theorem}

\begin{proof}
We first consider $\bigdw{\law{T_{a,n}}}{\cpspec{\theta}{n}}$, for which we
bound the quantities appearing in~\Ref{AB-Stein-bnd}, as addressed in
\Ref{AB-Stein-1}--\Ref{AB-Stein-3}.
The contribution from~\Ref{AB-Stein-3} is immediate.
Then, defining
\eqs
  \th^* &:=& \max\{1,\th,\sup_{i\ge1}\th_i\}\quad\mbox{and}\quad 
    \s^*_n \Def \sn \max\Blb \m_i,\frac1{ir_i} \Brb, 
\ens
we can easily bound the third element in~\Ref{AB-Stein-1}
by $\th^* \s^*_n \|g_f\|$,  and the second, using~\Ref{AB-Stein-2},
contributes at most $4\th^* \s^*_n \|g_f\|$, since also $ir_i \ge 1$.  For the
first term, we use Lemma~\ref{heartlemma}(i) to give
\[
  \Blm \sn (\th- \th_i) \ex g_f(T_{a,n}+i)\Brm \Le
   \{2 \th^* m + n\tthnm +
          (1/4) \th^* mn D^1(T_{a,n})\}\|g_f\|.  
\]
In all, and using $\|g_f\| \le 1$, this gives the bound
\eq\label{AB-eps0}
    \bigdw{\law{T_{a,n}}}{\cpspec{\theta}{n}} \Le n\e_1(n,a,m),
\en
with
\eq\label{AB-eps1}
   \e_1(n,a,m) \Def \frac14 \th^* m D^1(T_{a,n}) + \tthnm + n^{-1}\th^*\{5 \th^* \s^*_n
   + 2m  + a\}\,.
\en

This bound, together with the inequality
\begin{equation*}
  \bigdw{\law{n^{-1}T_{a,n}}}{\ddist} \Le n^{-1} \bigdw{\law{T_{a,n}}}{\cpspec{\theta}{n}}
    + \bigdw{n^{-1}\cpspec{\theta}{n}}{\ddist} \, ,
\end{equation*}
now give the required estimate, since
\begin{equation*}
   \bigdw{n^{-1}\cpspec{\theta}{n}}{\ddist} \Le n^{-1}(1+\theta)^2 \, ;
\end{equation*}
see [ABT, Theorem 11.10].

If QLC holds, $D^1(T_{a,n}) = O(\{(a+1)/n\}^\g)$ for some $\g > 0$, and choosing $m
= m_n$ tending to infinity slowly enough ensures that $\e_1(n,a_n,m_n)\to 0$.
If QLC2 holds, choose~$m$ to be an appropriate power of $\{(a+1)/n\}$.
\end{proof}

With a little more difficulty, one can prove the analogous local approximation to
the distribution of~$T_{a,n}$. This the main tool for establishing the asymptotic
behaviour of quasi-logarithmic combinatorial structures. 

\begin{theorem}\label{local dickman approx theorem}
For any $0 \le a \le n$ and any $r \ge 2a+1$, we have 
\begin{equation}\label{llt equation}
   | n \bigpp{T_{a,n}=r} - \ddens(r/n) | \Le \min_{1\le m\le n}\e_5(n,a,m;r) \, ,
\end{equation}
with $\e_5(n,a,m;r)$ as defined in~\Ref{AB-eps5} below.  
If QLC holds, it follows that \break
$\sup_{r\ge nx}| n \bigpp{T_{a,n}=r} - \ddens(r/n) | \to 0$ for any $x>0$ and any
sequence $a_n = o(n)$.  If QLC2 holds, then
$\sup_{r\ge nx}| n \bigpp{T_{a,n}=r} - \ddens(r/n) | = O(\{(a+1)/n\}^{\h_2})$,
for any $x>0$ and for some $\h_2 > 0$.
\end{theorem}

\begin{proof}
With $x := r/n$, we begin by writing
\begin{equation*}
\begin{split}
   \bigeuknorm{\ddens(x) - n \pp{T_{a,n}=r} } & \Le 
      \frac1x \bigeuknorm{ \theta \pp{r-n \leq T_{a,n} < r - a} - r \pp{T_{a,n}=r} }\\
    & \mbox{}\quad 
      + \Bigeuknorm{ \ddens (x) - \frac1x\theta \bigpp{r-n \leq T_{a,n} < r - a} }\, . 
\end{split}
\end{equation*}

Now the quantity 
\[
    \D_1(r) \Def \theta \pp{r-n \leq T_{a,n} < r - a} - r \pp{T_{a,n}= r}
\]
is of the form $\ex\Blb  \th\sian g(T_{a,n}+i) - T_{a,n} g(T_{a,n}) \Brb$,
as in~\Ref{AB-Stein-1}, with $g := \bone_{\{r\}}$.  Take~$l_0$  such that
$\pr[Z_{l1} = 0] \ge 1/2$ for all $l \ge l_0$, and set 
\hbox{$C(l_0) := \{\min_{1\le l\le l_0}\max_{j \ge 1}\pr[Z_{l1} = j]\}^{-1}$;}
then we have
\eq\label{AB-prob-bnds}
   \pr[T_{a,n}\ui = s] \Le 2 \pr[T_{a,n} = s],\ i\ge l_0;\qquad
     \pr[T_{a,n}\ui = s] \Le C(l_0) \sup_{j\ge1} \pr[T_{a,n} = j]\,,
\en
for all $s\ge0$ and $1\le i < l_0$.
Note also that, by considering expectations of functions of the form~$\bone_{[0,j]}$,
\eq\label{AB-point-bnd}
   \sup_{j\ge1} \pr[T_{a,n} = j] \Le D^1(T_{a,n}).
\en 
Using these bounds, we can bound the third element in~\Ref{AB-Stein-1} by
\[
   \th^* \Blb \sum_{i=1}^{l_0-1} C(l_0)\m_i D^1(T_{a,n}) 
      + \sum_{i=l_0}^{l-1} 2\m_i D^1(T_{a,n}) + 2\m_l\Brb,
\]
for any $l\ge l_0$, since $\sum_{i = l}^n \pr[T_{a,n} = r - ik] \le 1$ for all $k\ge1$.
The second element is bounded, using~\Ref{AB-Stein-2}, in a very similar way, giving
\[
    2\th^* \Blb \sum_{i=1}^{l-1} (C(l_0)\vee2)\frac1{ir_i}(1+\m_i) D^1(T_{a,n}) 
       + \frac2{lr_l}(1+\m_l)\Brb.
\]
Finally, the first element in~\Ref{AB-Stein-1} is bounded by Lemma~\ref{heartlemma}(ii) as
\[
   \Blm \sian (\th-\th_i)\pr[T_{a,n}=r-i] \Brm \Le \tthnm
      + m\th^*\Bl 2 + \frac{m}6 \Br D^1(T_{a,n}).
\]
Combining these estimates, we conclude that, for any $l\ge l_0$,
\eq\label{AB-1st-bnd}
   \Blm \D_1(r) \Brm \Le \e_2(n,a,m),
\en
where
\eqa
  \lefteqn{\e_2(n,a,m) \ :=\ }\non\\ 
    && \th^* \min_{l \ge l_0}\Blb \sum_{i=1}^{l-1} (C(l_0)\vee2)
            \Bl\frac2{ir_i}(1+\m_i)+\m_i\Br D^1(T_{a,n}) 
       + 4\Bl \frac1{lr_l}(1+\m_l) + \m_l \Br\Brb \non\\
   &&\qquad\mbox{} + m\th^*\Bl 2 + \frac{m}6 \Br D^1(T_{a,n}) + \tthnm\,.
   \label{AB-eps2}
\ena

\sms
The next step is to bound the difference 
\[
    \D_2(r) \Def \bigpp{r-n \leq T_{a,n} < r - a} - \CP(\th,n)\{[r-n , r - a-1]\},
\]
which can once again be accomplished by using \Ref{AB-Stein-CP} and~\Ref{AB-Stein-bnd}.
Since, for $f := \bone_{[0,s-1]}$,
\[
   \|g_f\| \Le (1+\th)/(s + \th),
\]
by [ABT, Lemma~9.3], it follows as in the proof of~\Ref{AB-eps0} in the previous theorem that
\eq\label{AB-K-1}
   |\pr[T_{a,n} < s] - \CP(\th,n)\{[0,s-1]\}| \Le s^{-1}(1+\th)n\e_1(n,a,m),
\en
for any $s \ge 1$.  For $2a < r \le n$, this gives 
\[
   |\D_2(r)| \Le (r-a)^{-1}(1+\th)n\e_1(n,a,m) \Le 2r^{-1}n (1+\th)\e_1(n,a,m).
\]
For $r > n$,  two differences as in~\Ref{AB-K-1} are needed. The first is just as before;
the second is bounded by
\eq\label{AB-eps3-1}
   \e_3(n,a,m;r) \Def
     \min\{(r-n)^{-1}(1+\th)n\e_1(n,a,m), \pr[T_{a,n} < r-n] + \CP(\th,n)\{[0,r-n-1]\}\},
\en
where the alternative is useful if~$r$ is close to~$n$.  Now
\eq\label{AB-CPsmall}
   \CP(\th,n)\{[0,j]\} \Le \prod_{i=j+1}^n \Po(\th i^{-1})\{0\}
     \Eq \exp\Blb - \sum_{i=j+1}^n \th i^{-1}\Brb \Le \Bl \frac{j+1}{n+1} \Br^\th.
\en
Rather similarly,
\eqs
   \lefteqn{\pr[T_{a,n} \le j] \Le \prod_{i=j+1}^n \{\pr[Z_{i1}=0\}^{r_i} 
            \Le \exp\Blb - \sum_{i=j+1}^n \th_i i^{-1}(1-\m_i)\Brb} \\
    &\le& \left[ \exp\Blb - \sum_{i=j+1}^n \Bigl(\frac{\th_i - \th}i \Bigr)\Brb
       \Bl \frac{j+1}{n+1} \Br^\th \right]^{1-\m_{j+1}} \Le
     \Blb 2e^{1+\th^*} \Bl \frac{j+1}{n+1} \Br^{\th - \d(j/2,\th)} \Brb^{1-\m_{j+1}},
\ens
from Lemma~\ref{5.3}, and this in turn gives
\eq\label{AB-Tsmall}
   \pr[T_{a,n} \le j] \Le 2e^{1+\th^*} \Bl \frac{(j\vee j_0)+1}{n+1} \Br^{\th/4},
\en
where $\m_{j+1} \le 1/2$ and $\d(j/2,\th) \le \th/2$ for all $j\ge j_0$.
Using \Ref{AB-CPsmall} and~\Ref{AB-Tsmall} in \Ref{AB-eps3-1}, and optimizing with
respect to~$r$, gives
\[
    \e_3(n,a,m;r) \Le  4e^{1+\th^*} \{\e_1(n,a,m)\}^{\th/(4+\th)} \ =:\ \e_4(n,a,m).
\]
Hence
\eq\label{AB-eps3}
    |\D_2(r)| \Le  2nr^{-1}(1+\th)\e_1(n,a,m) + \e_4(n,a,m)
\en
for all $r \ge 2a+1$.

\sms
The remainder of the estimate is concerned with comparing the density~$\ddens(r/n)$
with $nr^{-1}\th\CP(\th,n)\{[r-n , r - a-1]\}$.  From [ABT, Theorem~11.12], it follows that
\eq\label{AB-3rd-bnd}
    |\CP(\th,n)\{[r-n , r - a-1]\} - P_\th\{[r/n-1 , (r - a)/n)\}| \Le c(\th)n^{-(\th\wedge1)},
\en
for a constant $c(\th)$, and then, from [ABT, (4.23) and~(4.20)],
\eqa
   \lefteqn{|nr^{-1}\th P_\th\{[r/n-1 , (r - a)/n)\} - \ddens(r/n)|}\label{AB-4th-bnd}\\
  &&\Eq nr^{-1}\th|P_\th\{[r/n-1 , (r - a)/n)\} - P_\th\{[r/n-1 , r/n)\}| 
   \Le c'(\th) nr^{-1} (a/n)^{(\th\wedge1)} \non
\ena
so long as $r \ge 2a$.  Combining 
\Ref{AB-1st-bnd}, \Ref{AB-eps3}, \Ref{AB-3rd-bnd} and~\Ref{AB-4th-bnd}, the theorem 
follows with
\eqa
  \e_5(n,a,m;r) &:=&  \frac nr \Bigl\{ 2\th(1+\th) nr^{-1}\e_1(n,a,m) + \e_2(n,a,m)\non\\
     &&\qquad\quad\mbox{} + \th  \e_4(n,a,m)
       + c''(\th)((a+1)/n)^{(\th\wedge1)} \Bigr\}; \label{AB-eps5}
\ena
note that, for $2a \le n/2 \le r\le n$, the bound can be replaced by the uniform
\eq\label{AB-eps5'}
    \e_5'(n,a,m) \Def 2 \Blb \e_2(n,a,m) +  4n^{-1}\th (1+\th)\e_1(n,a,m)
       + c''(\th)((a+1)/n)^{(\th\wedge1)} \Brb\,.
\en

\sms
If QLC holds, $D^1(T_{a,n}) = O(\{(a+1)/n\}^\g)$ for some $\g > 0$, and choosing $m
= m_n$ tending to infinity slowly enough ensures that $\e_l(n,a_n,m_n)\to 0$
for $l=1,2$ and~$4$; this implies that $\e_5(n,a_n,m_n,r)\to 0$ uniformly 
in $r\ge nx$, for any $x>0$.
If QLC2 holds, choose~$m$ to be an appropriate power of $\{(a+1)/n\}$.
\end{proof}

\section{Quasi-logarithmic structures}\label{sec-4}
In this section, we consider the two common properties shared by logarithmic
combinatorial structures that were discussed in the Introduction, and show that
they are also true  for quasi-logarithmic structures.  For each of the properties,
the local approximation of $\pr[T_{a,n}=r]$ in Theorem~\ref{local dickman approx theorem}
is the fundamental relation from which everything else follows.  Other aspects
of the asymptotic behaviour of logarithmic combinatorial structures could be extended to
quasi-logarithmic structures by analogous methods.

\subsection{The size of the largest component}

The following theorem is an extension of a result proved by \ocite{kingman:77} 
in the case of $\theta$-tilted random permutations.
A version for logarithmic structures can be found in [ABT, Theorem~7.13].

\begin{theorem}\label{largest component theorem}
Let
\begin{equation*}
   L^{\scriptscriptstyle (n)}  \Def \max \bigpset{1 \leq i \leq n}{C_i^{\scriptscriptstyle (n)} > 0}
\end{equation*}
be the size of the largest component. Then, if QLC holds,
\begin{equation*}
   \lim_{n \to \infty} \biglaw{n^{-1} L^{\scriptscriptstyle (n)}}  \Eq \law{L} \, ,
\end{equation*}
where $L$ is a random variable concentrated on $(0,1]$, whose distribution is given 
by the density function
\begin{equation*}
   f_{\theta}(x) \Def e^{\gamma \theta} \Gamma(\theta+1) x^{\theta-2} \: \! 
       \ddens\bigl((1-x)/x\bigr) \, , \qquad \text{for all $x \in (0,1]$.}
\end{equation*}
In particular, if $\theta=1$,
\begin{equation*}
    \lim_{n \to \infty} \bigpp{L^{\scriptscriptstyle (n)} \leq n/y}  \Eq \rho(y) \, , 
        \qquad \text{for all $y \geq 1$,}
\end{equation*}
where $\rho$ is Dickman's function \cite{dickman:30}.
\end{theorem}

\begin{proof} Fix $x \in (0,1]$. Then
\begin{equation}\label{AB-big-1}
  \bigpp{n^{-1} L^{\scriptscriptstyle (n)} \leq x} 
    \Eq \bigpp{C_{\lfloor nx \rfloor+1}^{\scriptscriptstyle (n)}= \dotsb 
           = C_n^{\scriptscriptstyle (n)}=0}
   \Eq \! \! \! \! \prod_{i=\lfloor nx \rfloor+1}^n \! \bigpp{Z_i=0} \: 
              \frac{\bigpp{T_{0,\lfloor nx\rfloor}=n}}{\bigpp{T_{0,n}=n}} \, .
\end{equation}
Theorem \ref{local dickman approx theorem} yields
\begin{equation}\label{AB-big-2}
     \frac{\bigpp{T_{0,\lfloor nx \rfloor}=n}}{\bigpp{T_{0,n}=n}}  
      \Eq \frac{n\ddens(n/\lfloor nx \rfloor)}{\lfloor nx \rfloor \: \! \ddens(1)} \,
       \Blb 1 + O(\min_m \e_5\bigl(\lfloor nx \rfloor,0,m;n) + \min_m \e_5(n,0,m;n)\bigr)\Brb .
\end{equation}
Writing $\theta_i:=i\ee Z_i$ and $y_i  := \theta_i(1+E_i)/(ir_i)$, where 
$E_i:=\sum_{k=1}^\infty \varepsilon_{ik}$, 
we obtain 
\begin{equation}\label{AB-big-3}
  \prod_{i=\lfloor xn \rfloor+1}^n \! \bigpp{Z_i=0}  \Eq
   \exp \biggl(- \! \! \! \! \sum_{i=\lfloor xn \rfloor+1}^n \frac{\theta_i}{i} \biggr)
   \exp \biggl(- \! \! \! \! \sum_{i=\lfloor xn \rfloor+1}^n \frac{\theta_i E_i}{i} \biggr)
   \prod_{i=\lfloor xn \rfloor+1}^n \biggl(\frac{1-y_i}{e^{-y_i}}\biggr)^{r_i} \, .
\end{equation}
From Lemma \ref{5.3}, 
\[
    \exp \biggl(- \! \! \! \! \sum_{i=\lfloor xn \rfloor+1}^n \frac{\theta_i}{i} \biggr)
       \Eq x^\th \{1 + O(\d(m,\th) + m/(nx))\},
\]
for any $m < \lfloor nx \rfloor/2$; then, easily,
\[
    \exp \biggl(- \! \! \! \! \sum_{i=\lfloor xn \rfloor+1}^n \frac{\theta_i E_i}{i} \biggr)
     \Eq 1 + O(\m_{\lfloor nx \rfloor})
\]
and
\[
    \prod_{i=\lfloor xn \rfloor+1}^n \biggl(\frac{1-y_i}{e^{-y_i}}\biggr)^{r_i}
     \Eq 1 + O(n^{-1}),
\]
so that $\prod_{i=\lfloor xn \rfloor+1}^n \! \bigpp{Z_i=0} \sim x^\th$ under QLC.
Combining this with~\Ref{AB-big-1} and~\Ref{AB-big-2}, it follows that then
\begin{equation}\label{AB-big-4}
  \lim_{n \to \infty} \bigpp{n^{-1} L^{\scriptscriptstyle (n)} \leq x}  \Eq
      \frac{x^{\theta} \ddens(1/x)}{x \: \! \ddens(1)} \ =:\ F_{\theta}(x) \, ,
\end{equation}
where $F_\theta$ is a distribution function with density $f_\theta$ [ABT, p.~108].
 If $\theta=1$, then $\ddens(x)=e^{-\gamma} \rho(x)$. This
proves the theorem.

Under QLC2, the convergence rate in~\Ref{AB-big-4} for each~$x>0$ is of order~$O(n^{-\h_3})$,
for some $\h_3 > 0$.
\end{proof}

One can also prove local versions of the convergence theorem.  However, they have to
involve the particular sequence~$\th_i$, since, for instance, $\pr[L\un = r] = 0$
if $\th_r=0$, because then $Z_r$, and hence also $C\un_r$, are zero a.s.  A typical
result is as follows.
 
\begin{theorem}
 If QLC holds, then, for any~$0 < x\le 1$ such that $1/x$ is not an integer,
it follows that
\eq\label{AB-local0}
  \lnti|n\pr[L\un = \nx] - (\th_{\nx}/\th)f_\th(x)| \Eq 0.
\en
Under QLC2, the convergence rate is of order~$O(n^{-\h_4})$, for some $\h_4 > 0$.
\end{theorem}

\begin{proof}
 Arguing as in the proof of the previous theorem,
\eq\label{AB-local1}
   \pr[L\un = \nx] \Eq \prod_{i=\lfloor nx \rfloor+1}^n \! \bigpp{Z_i=0} \: 
           \sum_{l=1}^{\lfloor n/\nx \rfloor} \pr[Z_{\nx} = l]
              \frac{\bigpp{T_{0,\lfloor nx\rfloor-1}=n-l\nx}}{\bigpp{T_{0,n}=n}}.
\en
Now, from Theorem~\ref{local dickman approx theorem}, the ratios
\[
    \frac{\bigpp{T_{0,\lfloor nx\rfloor-1}=n-l\nx}}{\bigpp{T_{0,n}=n}}
\]
are bounded as $\nti$, uniformly for all $2\le l \le \lfloor n/\nx \rfloor$, 
provided that $1/x$ is not an integer, so that $1 - x\lfloor 1/x \rfloor> 0$.  Then
\eqs
   \sum_{l\ge2} \pr[Z_\nx = l] &\le& r_\nx\pr[Z_{\nx,1} \ge 2]
             + {r_\nx \choose 2}(\pr[Z_{\nx,1} = 1])^2 \\
   &\le& \frac1\nx \m_\nx + \frac{(\th^*)^2}{2\nx^2},
\ens
implying that $\lnti n\sum_{l\ge2} \pr[Z_\nx = l] = 0$.  Hence the sum of the terms
for $l\ge2$ on the right hand side of~\Ref{AB-local1} contributes an asymptotically
negligible amount to the quantity $n\pr[L\un = \nx]$ as $\nti$.  For the $l=1$
term in~\Ref{AB-local1}, both the product and
the ratio of point probabilities are treated as in the proof of 
Theorem~\ref{largest component theorem}, giving the limit 
$x^{\th-1}p_\th((1-x)/x)/p_\th(1)$, and
\[
    |\pr[Z_\nx = 1] - \th_\nx/\nx| \Le \m_\nx/\nx,
\]
so that
\[
 \lnti |n\pr[L\un = \nx] - x^{\th-1}\{p_\th((1-x)/x)/p_\th(1)\}\,x^{-1}\th_\nx| \Eq 0.
\]
This completes the proof of~\Ref{AB-local0}.  The remaining statement follows as
usual, by taking greater care of the magnitudes of the errors in the various
approximation steps.
\end{proof}

In order to relax the condition that $1/x$ is not integral, it is necessary to
strengthen the assumptions a little; for example, if $x=1/2$ and~$n$ is even, 
the contribution from the $l=2$ term in~\Ref{AB-local1} is of order
$O(\e_{n/2,2}n^{1-\th})$, which could be large for $\th<1$.  In order to get
a limit of $f_\th(x)$ without involving the individual values~$\th_i$, it is necessary 
to average the point probabilities over an interval of integers around~$\nx$,
of a length that grows with~$n$, but is itself of magnitude~$o(n)$.

\subsection{The spectrum of small components}

We prove an analogue of the Kublius fundamental lemma \cite{kubilius:64} for
quasi-logarithmic structures, and thus extend results of \ocite{ast:95} and [ABT, Theorem 7.7];
see also the corresponding result of Manstavi\v cius~(2009), proved under different
conditions.

\begin{theorem}\label{kubilius fundamental theorem}
For $a/n \le \a_0$, where~$\a_0$ is small enough that $\min_{m\ge1}\e'_5(n,a,m) \le \half\ddens(1)$,
we have
\begin{equation}\label{kubilius fundamental ineq}
  \bigdtv{\law{\cn_1,\dotsc,\cn_a}}{\law{Z_1,\dotsc,Z_a}}
   \ \leq\ \e_6(n,a) \, ,
\end{equation}
where the order of magnitude of $\e_6(n,a)$ is given in~\Ref{AB-small-1} below.

If QLC holds, then
\begin{equation*}
   \lim_{n \to \infty} \bigdtv{\law{\cn_1,\dotsc,\cn_{a_n}}}{\law{Z_1,\dotsc,Z_{a_n}}} \Eq 0
\end{equation*}
for every non-negative integer sequence $a_n=\lo{n}$.  If QLC2 holds, the convergence
rate is of order $\{(a+1)/n\}^{\h_5}$ for some~$\h_5 > 0$.
\end{theorem}

\begin{proof}
The proof is similar to that of [ABT, Theorem~5.2]. We fix an~$n$ with 
the required properties, and we set
$p_k := \pp{T_{a,n}=k}$. Then the conditioning relation entails
\begin{equation}
 \begin{split}
    & \bigdtv{\law{\cn_1,\dotsc,\cn_a}}{\law{Z_1,\dotsc,Z_a}}\\
    & \quad \Eq \sum_{k=1}^n \pp{T_{0,a}=k} \frac{\bigl(\pp{T_{0,n}=n}-p_{n-k}\bigr)^+}{\pp{T_{0,n}=n}}\\
    & \quad \Le \sum_{k =0}^{\lfloor n/2 \rfloor} \sum_{l=0}^{\lfloor n/2 \rfloor} 
         \pp{T_{0,a}=k} \pp{T_{0,a}=l} \frac{\bigl(p_{n-l}-p_{n-k}\bigr)^+}{\pp{T_{0,n}=n}}\\
    & \qquad \mbox{} + \sum_{k =0}^{\lfloor n/2 \rfloor} \sum_{l=\lfloor n/2 \rfloor+1}^n 
         \pp{T_{0,a}=k} \pp{T_{0,a}=l} \frac{p_{n-l}}{\pp{T_{0,n}=n}} 
       + \sum_{k=\lfloor n/2 \rfloor+1}^n \pp{T_{0,a}=k} \, .
\end{split}\label{AB-dtv-1}
\end{equation}
We now separately bound the three terms in~\Ref{AB-dtv-1}. 

The first term is just
\begin{equation}\label{a1 ineq 1}
    \sum_{0\leq k < l\leq n/2} \pp{T_{0,a}=k} \pp{T_{0,a}=l}
      \frac{\bigeuknorm{p_{n-l}-p_{n-k}}}{\pp{T_{0,n}=n}} \, .
\end{equation}
Now, from Theorem~\ref{local dickman approx theorem}, using the bound given in~\Ref{AB-eps5'},
we have
\eq\label{AB-dtv-2}
   | n p_{n-r} - \ddens(1-r/n) | \Le  \e'_5(n,a,m) \, ,\quad 0 \le r \le n/2,
\en
so that, in~\Ref{a1 ineq 1},
\[
   n|p_{n-l}-p_{n-k}| \Le |\ddens(1-l/n) - \ddens(1-k/n)| + 2\e_5'(n,a,m)
     \Le c_1(\th)n^{-1}|k-l| + 2\e_5'(n,a,m),
\]
for any choice of~$m$ and for some constant~$c_1(\th)$.  Since also, from
Theorem~\ref{local dickman approx theorem}, $n\pr[T_{0,n}=n]$ is uniformly
bounded below whenever $\e_5'(n,a,m) \le \half\ddens(1)$, it follows that the
first term in~\Ref{AB-dtv-1} is of order 
\eq\label{AB-dtv-3}
   O\bigl(n^{-1}\ex T_{0,a} +  \min_{m\ge1}\e_5'(n,a,m)\bigr) 
      \Eq O\bigl(n^{-1}a + \min_{m\ge1}\e_5'(n,a,m)\bigr).
\en

For the second term in~\Ref{AB-dtv-1}, we have two bounds.
First,
\begin{equation*}
  \sum_{l=\lfloor n/2 \rfloor + 1}^n \pp{T_{0,a}=l} \frac{\pp{T_{a,n}=n-l}}{\pp{T_{0,n}=n}} \Le
     \frac{n \max_{n/2 \le l \le n}\pp{T_{0,a}=l}}{n\pp{T_{0,n}=n}},
\end{equation*}
where the denominator is uniformly bounded below whenever $\e_5'(n,a,m) \le \half\ddens(1)$,
and, for $n/2 \le l \le n$ and $a\le n/4$,
\eq\label{AB-dtv-4}
   n\pr[T_{0,a} = l] \Le 2l\pr[T_{0,a}=l] \Le 2\th^* \pr[T_{0,a} \ge n/4] + 2\e_2(a,0,m) 
     \Le 8(\th^*)^2 (a/n) + 2\e_2(a,0,m)\,,
\en 
from~\Ref{AB-1st-bnd}, with the last step following because $\ex T_{0,a} \le a\th^*$.
The second bound is given by
\eq\label{AB-dtv-5}
 \begin{split}
   \sum_{l=\lfloor n/2 \rfloor + 1}^n \pp{T_{0,a}=l} \frac{\pp{T_{a,n}=n-l}}{\pp{T_{0,n}=n}}
    & \Le \pp{T_{0,a} \geq n/2} \frac{\sup\nolimits_{j \in \zzp}\pp{T_{a,n}=j}}{\pp{T_{0,n}=n}}\\
    & \leq \frac{2a \th^* D^1(T_{a,n})}{n\pp{T_{0,n}=n}} \, ,
 \end{split}
\en
again using Markov's inequality, and the asymptotically important part is $aD^1(T_{a,n})$.
Thus the second term in~\Ref{AB-dtv-1} is of order
\eq\label{AB-dtv-6}
   O\bigl(n^{-1}a +  \min\{\min_{m\ge1}\e_2(a,0,m),aD^1(T_{a,n})\}\bigr).
\en
Finally, the third term in~\Ref{AB-dtv-1} can be simply bounded from above by
\begin{equation}\label{AB-dtv-7}
   2 n^{-1}\ee T_{0,a}  \Le 2\th^*  n^{-1}a \, .
\end{equation}
Combining \Ref{AB-dtv-3}, \Ref{AB-dtv-6} and~\Ref{AB-dtv-7} proves the first part of the theorem,
with
\eq\label{AB-small-1}
    \e_6(n,a) \Eq O\bigl(n^{-1}a + \min_{m\ge1}\e_5'(n,a,m) 
        + \min\{\min_{m\ge1}\e_2(a,0,m),aD^1(T_{a,n})\}\bigr).
\en
The remaining statements follow as usual.
\end{proof}

\subsection{Additive arithmetic semigroups}
We now return to the example given in the Introduction, of a quasi-logarithmic 
combinatorial structure that is not logarithmic.  
In Knopfmacher's~(1979) additive arithmetic semigroups,
the elements of norm~$n$ can be decomposed into prime elements,
with $C_i\un$ the number having norm~$i$. The joint distribution
of $(C_1\un,\ldots,C_n\un)$ satisfies the conditioning relation, with
$Z_i \sim \nb(p(i),q^{-i})$, so that
\begin{equation*}
    \bigpp{Z_i=k} := \binom{p(i)+k-1}{k}q^{-ik} (1-q^{-i})^n \,,\quad k \in \zzp,
\end{equation*}
with the convention that $Z_i=0$ if $p(i)=0$.  Here, $p(i)$ denotes the number of
prime elements of norm~$i$, and~$q>1$ enters through the assumption that the number~$g(n)$ of
elements of size~$n$ satisfies
\begin{equation}\label{beurling condition}
    g(n) = q^n \sum_{j=1}^r c_j \: \! n^{\rho_j-1} + \bigbo{q^n n^{-\gamma}} \, ,
\end{equation}
for real numbers $\rho_1<\dotsb < \rho_r$ and $c_1, \dotsc, c_r$, with $\rho_r>0$, $c_r>0$,
and with $\gamma>1$, an analogue of a condition under which \ocite{beurling:37}
examined prime number theorems of so called generalized integers. 
In particular, \ocite{zhang:96}*{Theorem 6.2} shows,
under condition \eqref{beurling condition} with $\gamma>2$, that 
$\th_i := i\ex Z_i = ip(i)q^{-i}/(1-q^{-i})=\theta'_i+\lo{1}$, where 
$\set{\theta'_i}_{i \in \nn}$ 
is the integer skeleton of a sinusoidal mixture function 
\begin{equation}
   \theta'_t := \theta + \sum_{l=1}^L \lambda_l \cos (2 \pi f_l t - \varphi_l) \, , 
      \qquad t \in \rr \, ,
\end{equation}
with $\theta:=\rho_r>0$, amplitudes $\lambda_l>0$ such that $\sum_{l=1}^L \lambda_l \leq \theta$, 
 non-integral
frequencies $f_l \in [0,\infty) \setminus \zzp$ and phases $0 \leq \varphi_l < 2 \pi$.
In many examples, the sum of cosines is empty $(L=0)$, and the structure logarithmic.
When this is the case, the asymptotic behaviour of the small and large components is as
described in the Introduction: see \ocite{abt:05} for these and other results.  
Here, we are interested in establishing asymptotics in the case when $L\ge1$.

First, note that the same sequence~$\th'_i$, for integral~$i$, is obtained, if each~$f_l$ 
is replaced by its
fractional part $f_l - \lfloor f_l \rfloor$, so that the values of~$f_l$ can be taken to
lie in $(0,1)$; and then that, if~$f_l > 1/2$, it can first be replaced by $f_l - 1$, 
and then by $1 - f_l$ if also $\f_l$ is replaced by~$-\f_l$, again without changing
the~$\th'_i$. Hence we may assume that $f_l \in (0,1/2]$ for all~$l$.

Clearly, for $L\ge1$, the sequence~$\th'_i$ is in general not convergent in the usual sense, but,
in view of the properties of trigonometric functions, 
\[
   \Blb \sum_{i=i_1+1}^{i_2} \cos(2 \pi f_l t - \varphi_l) \Brb \Le \frac1{\sin \pi f_l},
\]
whatever the values of~$i_1,i_2$, so that $\tthnm = O(m^{-1}) \to 0$ as $m \to \infty$; and
$\m_i = O(q^{-i})$ as $i\to\infty$.  Hence the condition QLC2 is satisfied by these
structures if, for some $r,s$ coprime, $k\in\nn$ and $\ps > 0$,
the set $\{i \colon\, \min(\th_i,\th_{i+d}) \ge 4\ps\}$ has at least one element in
each $k$-interval $\{jk+1,\ldots,(j+1)k\}$ for all~$j$ sufficiently large and for all
$d \in D := \{r\}\cup\{s2^g,\,g\in\zzp\}$.

Now, if $\sum_{l=1}^L \l_l < \th$,
all the~$\th'_i$ are uniformly bounded below by $\ps_1 := \th - \sum_{l=1}^L \l_l > 0$,
and the condition QLC2 is clearly satisfied with any choice of $r,s$ and~$k$ if $\ps := \ps_1/8$,
because then all~$i$ sufficiently large are such that $\th_i \ge 4\ps$.

If $\sum_{l=1}^L \l_l = \th$, define 
\[
   V_l(i) \Def \min_{n \in \zz}|2 \pi f_l i - \varphi_l - (2n+1)\pi|,
\]
and observe that, if 
\[
    V_l(i) \Le \d_l \Def \pi\min\{f_l,1-2f_l\},
\]
then $|V_l(i+1)|$ and $|V_l(i+1)|$ are both at least~$\d_l$. Setting
\[
   L_i \Def \{1\le l\le L\colon\, V_l(i) \ge \d_l\},
\]
it also follows from the inequality $1 - \cos x \ge c x^2$ in $|x| \le \pi/3$,
with $2c = 1 - \pi^2/108$, that
\[
    \th'_i \ \ge\ c\sum_{l\in L_i} \l_l \d_l^2 \ =:\ 2\ps_2.
\]
If, for some~$i$, $\th'_i \le \ps_2$, it follows from the above considerations
that $\th'_{i+1}$ and $\th'_{i+2}$ are both at least~$\ps_2$, and hence every
interval of length $k=3$ far enough from the origin contains an index~$i$ 
with $\min(\th_i,\th_{i+d}) \ge 4\ps$, if $\ps := \ps_2/8 $, whatever the value of~$d$.  
Thus the condition
QLC2 is then satisfied for any choice of $r,s$ co-prime, provided that $\ps_2 > 0$.

There remains the possibility that $f_l = 1/2$ for all $1\le l\le L$, in which
case~$\ps_2=0$.  If any of the~$\f_l$ are not multiples of~$\pi$, the function~$\th'_i$
is once again uniformly bounded away from~$0$, and the same is true if one is an
even multiple of~$\pi$ and another an odd one.  Hence there are only two cases in which QLC2 is
not satisfied:
\[
   \th'_t \Def \th\{1 + \cos(\pi t)\} \quad\mbox{and}\quad \th'_t \Def \th\{1 + \cos(\pi (t-1))\}.
\]
In the former case, $\th'_i > 0$ only for {\it even\/}~$i$, and if $\pr[Z_i = 0] = 1$
for all odd~$i$ then $\dtv{\law{T_{0,n}}}{\law{T_{0,n}+1}} = 1$ for all~$n$; hence,
for instance, Theorem~\ref{local dickman approx theorem} cannot be expected to be true.
In the latter case, we can take $D := \{2^{g+1},\,g\ge0\}$, and use Theorem~\ref{AB-extra}
to show that $\dtv{\law{T_{a,n}}}{\law{T_{a,n}+1}} = O(\{(a+1)/n\}^\h)$ for some~$\h > 0$;
the rest of the argument is then as before.

\section{Technical bounds}
Here we collect some technical results that are needed to smooth out the
irregularities in the sequence of~$\th_i$'s. Let $\theta>0$, and let
$\set{\theta_i}_{i \in \nn}$ be any non-negative sequence. Let 
$\ths := \sup\nolimits_{i \in \nn} \theta_i$ and
$\thd :=\theta \vee \ths$. For every $m,n \in \nn$ we set
\begin{equation*}
    \tthnm := \sup_{j\ge0} 
     \: \Bigeuknorm{\frac{1}{m} \sum_{i=1}^{m} \theta_{j m + i} - \theta} \, .
\end{equation*}

\begin{lemma}\label{technical heart lemma}
Let $\set{y_i}_{i \in \nn}$ be a real-valued sequence, and $0 \le l < n$. Then
\eqs
  && \Blm \siln(\th_i-\th)y_i \Brm \Le \Blb
      \begin{array}{l}
         (2m\thd + n\tthnm)\|y\| + \ths(nm/8)\|\D y\|;\\[1ex]
         2m\thd\|y\| + \tthnm\sn|y_i| \\[1ex]
          \qquad\mbox{} + 
         m^{-1}{\ths} \sum_{l=1}^m \sum_{l'=1}^m \sum_{j=1}^{\lfloor n/m \rfloor - 1}
            |y_{jm+l} - y_{jm+l'}|,
      \end{array} \right.
\ens
where $\|y\| := \max_{l<i\le n}|y_i|$ and $\|\D y\| := \max_{l < i < n}|y_{i+1}-y_i|$.
\end{lemma}

\begin{proof}
For any $0 \le j \le \lfloor n/m \rfloor$, we have
\eqs
  \Blm \sum_{i=jm+1}^{(j+1)m} (\th_i-\th)y_i \Brm
   &=& \Blm \sum_{i=jm+1}^{(j+1)m}\{(\bth\uj - \th)y_i + (\th_i-\bth\uj)(y_i - \by\uj)\}\Brm \\
   &\le&  
     |\th-\bth\uj|\sum_{i=jm+1}^{(j+1)m}|y_i| + \ths \sum_{i=jm+1}^{(j+1)m} |y_i - \by\uj|,
\ens
where $\bth\uj := m^{-1}\sum_{i=jm+1}^{(j+1)m}\th_i$ and 
$\by\uj := m^{-1}\sum_{i=jm+1}^{(j+1)m} y_i$.
Note also that 
\[
    \sum_{i=jm+1}^{(j+1)m} |y_i - \by\uj| \Le \Blb
         \begin{array}{l}
               (m^2/8)\|\D y\|;\\[1ex]
               m^{-1}\sum_{l=1}^m \sum_{l'=1}^m |y_{jm+l} - y_{jm+l'}|.
         \end{array} \right.
\]
The lemma now follows by bounding the sum $\siln(\th_i-\th)y_i$ in $m$-blocks.
\end{proof}

\medskip
\begin{lemma}\label{heartlemma}
Let $T$ be an $\zzp$-valued random variable, and $0 \le l < n$, $m\ge1$. 

{\rm (i)} For every bounded function $g\colon\zzp \to \rr$, we have
\eq\label{heartlemma ineq 1}
   \biggeuknorm{\sum_{i=l+1}^n (\theta_i-\theta) \ee g(T + i)}
      \Le \supnorm{g} \biggl( 2 \thd m + n\tthnm +
          \frac14 \ths mn \: \bigdtv{\law{T}}{\law{T+1}}  \biggr) \, .
\en

{\rm (ii)} For every $k \in \nn$, we have
\begin{equation*}\label{heart coro ineq 2}
 \biggeuknorm{\sum_{i=l+1}^n(\theta_i-\theta) \pp{T+i=k}}
   \Le  \tthnm + m\Bl 2 \thd  + \frac16 \ths  m \Br \bigdtv{\law{T}}{\law{T+1}}  \, .
\end{equation*}
\end{lemma}

\begin{proof}
(i) We apply the first inequality in Lemma \ref{technical heart lemma}, 
with $y_i:=\ee g(T + i)$, noting that $\|\D y\| \le 2\bigdtv{\law{T}}{\law{T+1}}$. 

\sms
(ii) We apply the second inequality in Lemma \ref{technical heart lemma}, 
with $y_i:=\pp{X+i=k}$, and observe that then
\[
    \sup\nolimits_{j \in \zzp} \pp{T=j} \Le \bigdtv{\law{T}}{\law{T+1}},
\]
as for~\Ref{AB-point-bnd}, and that
\[
   \sum_{j=1}^{\lfloor n/m \rfloor - 1} |y_{jm+l} - y_{jm+l'}| \Le
       |l-l'|\bigdtv{\law{T}}{\law{T+1}}.
\]
\end{proof}

\medskip
\begin{lemma}\label{5.3}
If $0 < 2m \le l \le n$ , then 
\begin{equation*}
  \exp\biggeuknorm{\sum_{i=l+1}^n \frac{\theta_i - \theta}{i}}
     \Le \exp\{2m(1+\ths)/l\} \Bigl(\frac nl \Bigr)^{\tthnm} 
      \Le e^{1+\ths} \Bigl(\frac nl \Bigr)^{\tthnm}\,.
\end{equation*}
\end{lemma}

\begin{proof}
Choosing $y_i:=1/i$, the second inequality in Lemma~\ref{technical heart lemma} gives
\eqs
   \biggeuknorm{\sum_{i=l+1}^n \frac{\theta_i - \theta}{i}}
     &\le& 2\frac ml + \tthnm\sum_{i=l+1}^n \frac1i 
       + \ths\sum_{j=\lfloor l/m \rfloor}^{\lfloor n/m \rfloor - 1} \frac1{j(j+1)} \\
     &\le& 2(1+\ths)\frac ml + \tthnm\sum_{i=l+1}^n \frac1i ,
\ens
and the lemma follows.
\end{proof}

\end{document}